\documentclass{article}

\usepackage{authblk}
\usepackage{amsmath, amssymb, amsthm}
\usepackage{graphicx,url,enumerate}
\usepackage[lined,boxed,commentsnumbered]{algorithm2e}

\newtheorem{theorem}{Theorem}[section]
\newtheorem{lemma}[theorem]{Lemma}
\newtheorem{proposition}[theorem]{Proposition}
\newtheorem{corollary}[theorem]{Corollary}
\theoremstyle{definition}
\newtheorem{definition}[theorem]{Definition}
\newtheorem{example}[theorem]{Example}
\newtheorem{remark}[theorem]{Remark}

\newcommand{\zco}{\mathfrak{C}\mspace{1mu}}
\newcommand{\zdu}{\mathfrak{D}\mspace{1mu}}
\newcommand{\CZ}{\mathbb{C}\mathfrak{Z}}

\newcommand{\Z}{\mathfrak{Z}}

\numberwithin{equation}{section}

\title{Spectrally Simple Zeros of Zeon Polynomials}
\author{G.~Stacey Staples}
\affil{Department of Mathematics \& Statistics\\
 Southern Illinois University Edwardsville\\
 Edwardsville, IL 62026-1653, USA\\
Email: sstaple@siue.edu}

\date{}

\begin{document}

\maketitle

\begin{abstract}
Combinatorial properties of zeons have been applied to graph enumeration problems, graph colorings, routing problems in communication networks, partition-dependent stochastic integrals, and Boolean satisfiability.   Power series of elementary zeon functions are naturally reduced to finite sums by virtue of the nilpotent properties of zeons.  Further, the zeon extension of any analytic complex function has zeon polynomial representations on associated equivalence classes of zeons.  

In this paper, zeros of polynomials over complex zeons are considered.  Existing results for real zeon polynomials are extended to the complex case and new results are established.  In particular, a fundamental theorem of zeon algebra is established for spectrally simple zeros of complex zeon polynomials, and an algorithm is presented that allows one to find spectrally simple zeros when they exist.  As an application, inverses of zeon extensions of analytic functions are computed using polynomial methods. \\
Keywords: zeons; polynomials; fundamental theorem of algebra; spectral equivalence classes; analytic functions\\
MSC2020: 13B25, 05E40, 15A66, 81R05
\end{abstract}

\section{Introduction}

Letting $Z_n$ denote the multiplicative semigroup generated by a collection of commuting null-square variables $\{\zeta_{\{i\}}:1\le i\le n\}$ and identity $1=\zeta_\varnothing$, the resulting $\mathbb{R}$-algebra,  denoted here by $\Z_n$ has appeared in multiple guises over many years.  In recent years, $\Z_n$ has come to be known as the {\em $n$-particle zeon\footnote{The name ``zeon algebra'' was coined by Feinsilver~\cite{Feinsilverzeons}, stressing their relationship to both bosons (commuting generators) and fermions (null-square generators).} algebra}.   They naturally arise as commutative subalgebras of fermions.

Combinatorial properties of zeons have been applied to graph enumeration problems, partition-dependent stochastic integrals, and routing problems in communication networks, as summarized in \cite{OCGraphs}.   More recent combinatorial applications include graph colorings~\cite{StaplesStellhorn2017} and Boolean satisfiability~\cite{DavisStaples2019}.   Combinatorial identities involving zeons have also been developed in papers by Neto~\cite{NetoJIS2014,NetoJIS2015,NetoJIS2016,NetoSIAM}.   
   
A permanent trace formula analogous to MacMahon's Master Theorem was presented and applied by Feinsilver and McSorley in \cite{FeinsilverMcSorleyIJC}, where the connections of zeons with permutation groups acting on sets and the Johnson association scheme were illustrated.  

Power series of elementary zeon functions are reduced to zeon polynomials by virtue of the nilpotent properties of zeons~\cite{staplesweygandt}.  The zeon extension of any analytic function is reduced to a polynomial of degree not exceeding the number of generators in the algebra.  Finding zeros of zeon polynomials is of immediate interest for zeon differential calculus based on the power series approach~\cite{zeoncalculus}.   

From a different perspective, differential equations of first and second order in zeon algebras have recently been studied by Mansour and Schork~\cite{Mansour}.  Combinatorial properties of zeons are prominent throughout their work.

In the current paper, zeros of polynomials over complex zeons are considered.  Existing results for real zeon polynomials are extended to the complex case and new results are established.  In particular, a fundamental theorem of algebra is established for spectrally simple zeros of complex zeon polynomials.  An algorithm is presented that allows one to find spectrally simple zeros when they exist.

The rest of the paper is laid out as follows.  Terminology and notational conventions are established in subsection \ref{preliminaries}.  Useful machinery for polynomial division is established in Section \ref{poly div}, where a zeon remainder theorem is developed.  Zeon polyomials with complex scalar coefficients are considered in Section \ref{complex coefficients}, where basic results on zeros and their cardinalities are established.  Essential results on $k$th roots of complex zeons are established in Section \ref{complex zeon roots}.

The main results of the paper appear in Section \ref{zeon FTA}, where a Fundamental Theorem of Zeon Algebra is presented alongside an algorithm for computing spectrally simple zeros of complex zeon polynomials.  The paper closes with concluding remarks in Section \ref{conclusion}.

Examples appearing throughout the paper were computed using {\em Mathematica} with the {\bf CliffMath}~\cite{CliffMath} package, which is freely available through the {\em Research} link on the author's web page, \url{http://www.siue.edu/~sstaple}.  For a broader perspective of combinatorial and algebraic properties and applications of zeons, the interested reader is further directed to the books \cite{OCGraphs} and \cite{CAZ}

\subsection{Preliminaries}\label{preliminaries}

For $n\in\mathbb{N}$,  let $\CZ_n$ denote the complex abelian algebra generated by the collection $\{\zeta_{\{i\}}: 1\le i\le n\}$ along with the scalar $1=\zeta_\varnothing$ subject to the following multiplication rules:
\begin{gather*}
\zeta_{\{i\}}\,\zeta_{\{j\}}=\zeta_{\{i,j\}}=\zeta_{\{j\}}\,\zeta_{\{i\}}\,\,\text{ \rm
for }i\ne j,\text{ \rm and}\\
{\zeta_{\{i\}}}^2=0\,\,\text{ \rm for }1\le i\le n.
\end{gather*}
It is evident that a general element $u\in\CZ_n$ can be expanded as $u=\displaystyle\sum_{I\in2^{[n]}}u_{I}\,\zeta_{I}$, or more simply as $\sum_I u_I\zeta_I$,  where $I\in2^{[n]}$ is a subset of the {\em $n$-set}, $[n]:=\{1,2,\ldots,n\}$, used as a multi-index, $u_{I}\in\mathbb{C}$, and $\displaystyle\zeta_{I}=\prod_{\iota\in I}\zeta_\iota$.  The algebra $\CZ_n$ is called the ($n$-particle) complex {\em zeon algebra}\footnote{The $n$-particle (real) zeon algebra has also been denoted by ${\mathcal{C}\ell_n}^{\rm nil}$ in a number of papers because it can be constructed as a subalgebra of the Clifford algebra $\mathcal{C}\ell_{n,n}$.}.  

As a vector space, this $2^n$-dimensional algebra has a canonical basis of  {\em basis blades} of the form $\{\zeta_I: I\subseteq [n]\}$.  The null-square property of the generators $\{\zeta_i:1\le i\le n\}$ guarantees that the product of two basis blades satisfies the following:\begin{equation}\label{blade product}
\zeta_I\zeta_J=\begin{cases}
\zeta_{I\cup J}& I\cap J=\varnothing,\\
0&\text{\rm otherwise.}
\end{cases}
\end{equation}
It should be clear that $\CZ_n$ is graded.  For non-negative integer $k$, the {\em grade-$k$ part} of element $u=\sum_I u_I\zeta_I$ is defined as \begin{equation*}
\langle u\rangle_k=\sum_{\{I:|I|=k\}}u_I\zeta_I.
\end{equation*}

The maximal ideal consisting of nilpotent zeon elements will be denoted by \begin{equation*}
{\CZ_n}^\circ=\{u\in \CZ_n: \zco u=0\}.
\end{equation*}   
The multiplicative abelian group of invertible zeon elements is denoted by \begin{eqnarray*}
{\CZ_n}^\times&=&\CZ_n\setminus {\CZ_n}^\circ\\
&=&\{u\in \CZ_n: \zco u\ne 0\}.
\end{eqnarray*}

For convenience, arbitrary elements of $\CZ_n$ will be referred to simply as ``zeons.''  In what follows, it will be convenient to separate the scalar part of a zeon from the rest of it. To this end, for $z\in\CZ_n$ we write $\zco z=\langle z\rangle_0$, the {\em complex (scalar) part} of $z$, and $\mathfrak{D}z=z-\zco z$, the {\em dual part}~\footnote{The term ``dual'' here is motivated by regarding zeons as higher-dimensional dual numbers.} of $z$.

\subsection{Multiplicative Properties of Zeons}\label{group properties}

Since $\CZ_n$ is an algebra, its elements form a multiplicative semigroup.  It is not difficult to establish convenient formulas for expanding products of zeons.   As shown in \cite{DollarStaples}, $u\in\CZ_n$ is invertible if and only if $\zco u\ne 0$.  Moreover, the multiplicative inverse of $u$ is unique.  The result is paraphrased here for review.  

\begin{proposition}\label{zeon inverses}
Let $u\in\CZ_n$, and let $\kappa$ denote the index of nilpotency~\footnote{In particular, $\kappa$ is the least positive integer such that $(\mathfrak{D}u)^\kappa=0$.} of $\mathfrak{D}u$.    It follows that $u$ is uniquely invertible if and only if $\zco u\ne0$, and the inverse  is given by \begin{equation*}
u^{-1}=\frac{1}{\zco u}\sum_{j=0}^{\kappa-1}(-1)^{j}(\zco u)^{-j}(\mathfrak{D}u)^{j}.
\end{equation*}
\end{proposition}

For $n\in \mathbb{N}$,  ${\CZ_n}^\times$ is defined to be the collection of invertible elements in $\CZ_n$.  That is,
\begin{equation*}
 {\CZ_n}^\times=\{u\in  \CZ_n: \zco u\ne 0\}.
 \end{equation*}
Since ${\CZ_n}^\times$ is closed under (commutative) zeon multiplication, has multiplicative identity $1$, and every element has a multiplicative inverse, the invertible zeons  form an abelian group under multiplication.  

\section{Complex Zeon Roots: Existence and Recursive Formulations}\label{complex zeon roots}

The material below is generalized from the original (real) context seen in \cite{DollarStaples}.   As will be shown, invertible zeons have roots of all orders.  A recursive algorithm establishes their existence and provides a convenient method for their computation.  

\begin{theorem}\label{zeon roots}
Let $w\in{\CZ_n}^\times$, and let $k\in\mathbb{N}$.  Then, $\exists u\in{\CZ_n}^\times$ such that $u^k=w$, provided $\zco w\ne0$.  Further, writing $w=\varphi+\zeta_{\{n\}} \psi$, where $\varphi, \psi\in\CZ_{n-1}$, $u$ is computed recursively by \begin{equation*}
u=w^{1/k}=\varphi^{1/k}+\zeta_{\{n\}} \frac{1}{k}\varphi^{-(k-1)/k}\psi.
\end{equation*} 
\end{theorem}

\begin{proof}
Fix $k\in\mathbb{N}$.  Assuming $w\in {\CZ_n}^\times$ ensures that  $w_\emptyset\ne 0$, so the scalar part of $w$ has $k$ distinct complex $k$th roots.  

Proof is by induction on $n$.  When $n=1$, let $w=a+b\zeta_{\{1\}}$, where $w_\emptyset=a$.  Let $a^{1/k}\in\mathbb{C}$ denote a fixed $k$th complex root of $a$.  Applying the binomial theorem and null-square properties of zeon generators, one finds 
\begin{equation*}
\left(a^{1/k}+\frac{b}{ka^{(k-1)/k}}\zeta_{\{1\}}\right)^k=a+ka^{(k-1)/k}\frac{b}{ka^{(k-1)/k}}\zeta_{\{1\}}=a+b\zeta_{\{1\}}.
\end{equation*}

Next, suppose the result holds for some $n-1\ge 1$ and let $w\in\CZ_n$ be written $w=\varphi+\zeta_{\{n\}} \psi$, where $\varphi, \psi\in\CZ_n$.  In particular,  this implies $\varphi\in{\CZ_n}^\times$. Let $\alpha=\varphi^{1/k}$, and let $u=\alpha+ \displaystyle\frac{1}{k}\zeta_{\{n\}}\alpha^{-(k-1)}\psi$.
Then \begin{eqnarray*}
u^k=\left(\alpha+\zeta_{\{n\}} \frac{1}{k}\alpha^{-(k-1)}\psi\right)^k
&=& \varphi+k\alpha^{(k-1)}\frac{1}{k}\zeta_{\{n\}}\alpha^{-(k-1)}\psi\\
&=& \varphi+\zeta_{\{n\}}\psi\\
&=&w.
\end{eqnarray*}
\end{proof}

As expected, any invertible complex zeon has $k$ distinct complex $k$th roots.

\begin{corollary}
Let $\alpha\in{\CZ_n}^\times$, and let $k\in \mathbb{N}$.  Then, $\alpha$ has exactly $k$ distinct $k$th roots; i.e., $\sharp\{u: u^k=\alpha\}=k$.
\end{corollary}

\begin{proof}
From the recursive formulation of Theorem \ref{zeon roots}, we see that a complex zeon $k$th root of $\alpha$ is obtained for each of the distinct $k$th roots of $\zco\alpha$.  It suffices to show that two $k$th roots of $\alpha$ are equal whenever their scalar parts are equal.  To that end, suppose $u^k=\alpha=v^k$, and observe that $u-v$ is nilpotent if and only if $\zco u = \zco v$.  Assuming $\zco u =\zco v$, $u$ and $v$ can be written in the form $u=a+\beta$ and $v=a+\gamma$ for some nilpotent $\beta$ and $\gamma$.  Observe that the product $\alpha\delta$ of an invertible element $\alpha$ and a nilpotent $\delta$, is zero if and only if  $\delta=0$, since $0=\alpha^{-1}0=\delta$.  Hence, assuming $u^k=v^k$, one finds
\begin{eqnarray*}
u^k-v^k&=&(u-v)(u^{k-1}+u^{k-2}v+\cdots+v^{k-1})\\
&=&(u-v)\left[(a^{k-1}+\delta_1)+(a^{k-1}+\delta_2)+\cdots+(a^{k-1}+\delta_{k})\right]\\
&=&(u-v)\left(ka^{k-1}+\delta\right),
\end{eqnarray*}
where $\delta=\delta_1+\cdots+ \delta_{k}$ is nilpotent by the ideal property of ${\CZ_n}^\circ$.  It is clear that $ka^{k-1}+\delta$ is invertible, so $(u-v)(ka^{k-1}+\delta)=0$ implies $(u-v)=0$. 
 \end{proof}

Given an invertible zeon $u$ and positive integer $k$, it now makes sense to define the {\em principal $k$th root} of $u$ as the zeon $k$th root of $u$ whose scalar part is the principal root of $\zco u\in\mathbb{C}$.

\section{Zeon Polynomial Division}\label{poly div}

When the leading coefficient of a zeon polynomial is invertible, polynomial division is possible in the familiar sense.  This is formalized as follows.

\begin{theorem}[Zeon polyomial division algorithm]\label{zeon poly division}
Let $\varphi(u)$ and $\psi(u)$ be polynomials with zeon coefficients such that the leading coefficient of $\psi(u)$ is invertible.  Then, there exist unique zeon polynomials $q(u)$ and $r(u)$ such that $0\le \deg r(u)<\deg \psi(u)$ and 
\begin{equation}\label{zeon polydiv}
\varphi(u)=\psi(u)q(u)+r(u).
\end{equation}
\end{theorem}

\begin{proof}
First, existence of the polynomials $q$ and $r$ is established.  Suppose $\varphi(u)=a_mu^m+\cdots+a_1u+a_0$ and $\psi(u)=b_ku^k+\cdots+b_1u+b_0$ where $\zco b_k\ne 0$.  When $k>m$, taking $q(u)=0$ and $r(u)=\varphi(u)$ clearly satisfies the statement of the theorem.  Proof is then by strong induction on $m$, the degree of $\varphi$.  Suppose $m=1$ and $k\le m$. Then,
\begin{equation*}
a_1u+a_0=\begin{cases}
(b_1u+b_0)(a_1{b_1}^{-1})+(a_0-b_0a_1{b_1}^{-1}) &k=1\\
b_0(a_1{b_0}^{-1}u+a_0{b_0}^{-1})&k=0.
\end{cases}
\end{equation*}
When $k=1$, taking $q(u)=a_1{b_1}^{-1}$ and $r(u)=a_0-b_0a_1{b_1}^{-1}$ satisfies \eqref{zeon polydiv}.  When $k=0$ existence is satisfied by taking $q(u)=a_1{b_0}^{-1}u+a_0{b_0}^{-1}$ and $r(u)=0$.

Assume now that for some $m\ge 2$, existence of $q$ and $r$ satisfying \eqref{zeon polydiv} has been established for all zeon polynomials of degree less than $m$.  Letting $q_1(u)=a_m{b_k}^{-1}u^{m-k}$, let $\rho(u)$ be defined by 
\begin{eqnarray*}
\rho(u)&=&\varphi(u)-\psi(u)q_1(u)\\
&=&(a_mu^m+\cdots+a_1u+a_0)-(b_ku^k+\cdots+b_1u+b_0)(a_m{b_k}^{-1}u^{m-k})\\
&=&(a_{m-1}u^{m-1}+\cdots+a_1u+a_0)-g(u)
\end{eqnarray*}
for zeon polynomial $g(u)$ satisfying $\deg g(u)<k\le m$.  Hence, $\deg \rho(u)<m$.  By the inductive hypothesis, there now exist zeon polynomials $q_2(u)$ and $r(u)$ such that $r_1(u)=\psi(u)q_2(u)+r(u)$, where $\deg r(u)<\deg \psi(u)$.  It follows that \begin{eqnarray*}
\varphi(u)&=&\psi(u)q_1(u)+r_1(u)\\
&=&\psi(u)q_1(u)+\psi(u)q_2(u)+r(u)\\
&=&\psi(u)(q_1(u)+q_2(u))+r(u).
\end{eqnarray*}
Setting $q(u)=q_1(u)+q_2(u)$ completes the existence proof.
Turning now to uniqueness of $q$ and $r$, suppose \begin{eqnarray*}
\varphi(u)&=&\psi(u)q_1(u)+r_1(u)\\
&=&\psi(u)q_2(u)+r_2(u)
\end{eqnarray*}
holds for zeon polynomials $q_i(u)$, $r_i(u)$, where $\deg r_i(u)< \deg \psi(u)$ for each  $i=1,2$.  It follows that 
\begin{equation}\label{unique q_r}
\psi(u)(q_1(u)-q_2(u))=r_2(u)-r_1(u).
\end{equation}
If $q_1(u)-q_2(u)\ne 0$, invertibility of the leading coefficient of $\psi(u)$ guarantees that the left-hand side of \eqref{unique q_r} is a zeon polynomial of degree $t\ge \deg\psi(u)$.  However, the right-hand side is a zeon polynomial of degree $s<\deg\psi(u)$, a contradiction.  Thus, $q_1(u)-q_2(u)=0=r_1(u)-r_2(u)$ and uniqueness of $q(u), r(u)$ is established.
\end{proof}

\begin{corollary}[Zeon remainder theorem]
Let $\varphi(u)$ be a polynomial with zeon coefficients and let $z\in{\CZ_n}$.  Then, $\varphi(z)$ is the remainder after dividing $\varphi(u)$ by the monic polynomial $u-z$.
\end{corollary}

\begin{proof}
The result follows from an application of Theorem \ref{zeon poly division}; in particular,
 \begin{equation*}
\varphi(u)=(u-z)q(u)+r(u).
\end{equation*}
\end{proof}

\section{Zeon Polynomials with Complex Coefficients}\label{complex coefficients}

Let $f(z)=a_mz^m+\cdots+a_1z+a_0$ $(a_m\ne 0)$ be a polynomial function with complex coefficients, and recall that a complex number $r$ such that $f(r)=0$ is called a {\em zero} of $f(z)$.  By the Fundamental Theorem of Algebra, $f(z)$ has exactly $m$ complex zeros.  

If $f(z)$ can be written in the form $f(z)=(z-r)^\ell g(z)$, where $m\in\mathbb{N}$ and $g(r)\ne0$, then $r$ is said to be a {\em zero of multiplicity $\ell$} of $f(z)$.  Equivalently, one sees that $r$ is a zero of multiplicity $\ell$ of $f(z)$ if and only if $f(r)=f'(r)=\cdots=f^{(\ell-1)}(r)=0$ and $f^{(\ell)}(r)\ne0$. Here $f^{(\ell)}(r)$ denotes the $\ell$th derivative of $f(z)$ evaluated at $r$.  For convenience, $\mu_f(r)$ will denote the multiplicity of $r$ as a zero of $f(z)$.

Letting $\varphi:\CZ_n\to \CZ_n$ denote the zeon extension\footnote{The zeon extension $\varphi$ is the polynomial function obtained by extending the domain of $f$ from $\mathbb{C}$ to $\CZ_n$.} of $f$, one is led to wonder about the number of zeros this polynomial may have in $\CZ_n$.

The easiest case to consider is the existence of nilpotent zeros of extended polynomials $\varphi(u)$ in $\CZ_n$.  As illustrated in the following lemma, a nonzero polynomial either has no nilpotent zeros or infinitely many of them.

\begin{lemma}\label{nilpotent zeros}
Let $f(z)$ be a nonzero complex polynomial, and let $d$ be the least nonnegative integer such that $f^{(d)}(0)\ne 0$.   Let $\varphi$ denote the zeon extension of $f(z)$ over $\CZ_n$, where $n\ge 1$.  If $d\in\{0,1\}$, then $\varphi$ has no nilpotent zeros.  If $d\ge 2$, then $\varphi$ has infinitely many nilpotent zeros in $\CZ_n$.
\end{lemma}

\begin{proof}
Assuming $d$ is the least nonnegative integer such that $f^{(d)}(0)\ne 0$, it is clear that $f(z)$ can be written as $f(z)=a_mz^m+a_{m-1}z^{m-1}+\cdots+a_dz^d$.  

If $d=0$ or $d=1$, let $u$ be any nilpotent element of $\CZ_n$.  For any positive integer $\ell$, the minimum grade of $u^\ell$ is either zero or at least $\ell\natural u$.   Hence, no cancellation of the nonzero scalar term can occur in the $d=0$ case, and the grade-$\natural u$ part of $u$ is preserved by the degree-$1$ terms of the polynomial in the $d=1$ case.   Thus, $\natural \varphi(u)\ge \natural u>0$.

When $d\ge 2$, let $u=a\zeta_I$ for nonzero scalar $a$ and multi-index $I\ne\varnothing$.    Then, $\varphi(u)=a_mu^m \cdots+a_du^d=0$.
\end{proof}

While nilpotent zeros provide an easy special case, discussion now turns to more general zeros of zeon polynomials.  
  
\begin{proposition}\label{complex zeros}
Let $f$ be a nonzero polynomial with complex coefficients, and let $\varphi$ be its zeon extension.  If $f$ has no complex zeros, then $\varphi$ has no zeros.  If $r$ is a complex zero of $f$, then any  zeon element $w$ satisfying $\zco w=r$ and $\kappa(\mathfrak{D}w)\le \mu_f(r)$ is a zero of $\varphi$.  Moreover, the zeros of $\varphi$ are 
\begin{equation*}
\varphi^{-1}(0)=\{w\in{\CZ_n}^\times: f(\zco w)=0, \kappa(\mathfrak{D}w)\le\mu_f(\zco w)\}.
\end{equation*}
\end{proposition}
  
\begin{proof}
First, if $f$ has no complex zero, then $\zco(\varphi(u))=f(\zco u)$ guarantees that $\varphi$ has no zero.  Otherwise, suppose $f$ has distinct complex zeros $r_1, \ldots, r_k$.  It follows that $f$ can be written multiplicatively as  $f(x)=\displaystyle\prod_{i=1}^k (x-r_i)^{\mu_f(r_i)}g(x)$, where $g(x)$ is a constant or polynomial over $\mathbb{R}$ having no complex zeros.  

Next, let $r$ be a complex zero of $f$ and label any remaining complex zeros as $r_1, \ldots, r_{k-1}$.  Let $u\in{\CZ_n}$ satisfy $\zco u=r$.  Writing $\rho_i=\zco u - r_i$ for each $i$, and letting $\gamma$ be the zeon extension of $g$, the multiplicative form of the zeon extension $\varphi$ of $f$ evaluated at $u$ can now be written as  
\begin{eqnarray*}
\varphi(u)&=&\displaystyle(u-r)^{\mu_f(r)}\prod_{i=1}^{k-1} (u-r_i)^{\mu_f(r_i)}\gamma(u)\\
&=&\displaystyle(\zco u-r+\mathfrak{D}u)^{\mu_f(r)}\prod_{i=1}^{k-1} (\zco u-r_i+\mathfrak{D}u)^{\mu_f(r_i)}\gamma(u)\\
&=& (\mathfrak{D}u)^{\mu_f(r)}\prod_{i=1}^{k-1}(\rho_i+\mathfrak{D}u)^{\mu_f(r_i)}\gamma(u).
\end{eqnarray*}
Note that as an element of ${\CZ_n}$, $\gamma(u)$ has a multiplicative inverse;  otherwise, $\gamma(\zco u)=g(\zco u)=0$, contradicting the assumption that $g$ has no complex zeros.  For each $i$,  $(\rho_i+\mathfrak{D}u)$ is similarly  invertible, so the product \hfill\break $\displaystyle\prod_{i=1}^{k-1}(\rho_i+\mathfrak{D}u)^{\mu_f(r_i)}g(u)$ is invertible.    Consequently, $\varphi(u)=0$ if and only if $(\mathfrak{D}u)^{\mu_f(r)}=0$.  In other words, $\varphi(u)=0$ if and only if both $f(\zco u)=0$ and $\kappa(\mathfrak{D}u)\le \mu_f(r)$ are true.   
\end{proof}
  
\begin{remark}\label{simple roots example}
If $r$ is a simple zero of $f$, then the only element $u\in{\CZ_n}$ satisfying $\zco u=r$ and $\varphi(u)=0$ is $u=r$.
\end{remark}

\begin{example}
The roots of $f(z)=z^2+z$ are $z=0$ and $z=-1$.  Since both are simple, Remark \ref{simple roots example} states that the only zeros of $\varphi(u)=0$ are $u=0$ and $u=-1$.  Clearly, these are solutions, and according to Proposition \ref{complex zeros}, any other zero should be of the form   $u=\mathfrak{D}u$ or $u=-1+\mathfrak{D}u$ where $\kappa(\mathfrak{D}u)\le1$.  Clearly, this holds only for $\mathfrak{D}u=0$.
\end{example}

An immediate corollary is obtained when $f$ has a zero of multiplicity two or greater.

\begin{corollary}
Let $f$ be a nonzero polynomial function with complex coefficients and let $\varphi$ be its zeon extension.  If $f(r)=f'(r)=0$ for some $r\in\mathbb{R}$, then $\varphi$ has infinitely many zeros in ${\CZ_n}$ $(n\ge 1)$.  
\end{corollary}

\begin{proof}
If $f(r)=f'(r)=0$, then $\mu_f(r)\ge 2$.  Letting $w=r+a\zeta_{\{1\}}$ for any nonzero $a\in\mathbb{R}$, Proposition \ref{complex zeros} implies $\varphi(w)=0$.
\end{proof}

\begin{example}
Consider the polynomial $f(x)=x^2+x+1/4$.   The nontrivial zeon solutions of $\varphi(u)=0$ in $\CZ_1$ are seen to be $\{a\zeta_{\{1\}}-1/2: a\ne 0\}$.
\end{example}

Proposition \ref{complex zeros} also implies that nonzero polynomials having complex zeros of multiplicity exceeding $n$ have infinitely many zeros in ${\CZ_n}$, since the condition $\kappa(\mathfrak{D}w)\le \mu_f(\zco w)$ is satisfied automatically when $\mu_f(\zco w)>n$.  More formally, we have the following corollary.

\begin{corollary}
Let $f$ be a nonzero polynomial function with complex coefficients and let $\varphi$ be its zeon extension.  If $r$ is a zero of multiplicity $m$ of $f$ and $n< m$, then $\varphi(w)=0$ for any $w\in{\CZ_n}$ such that $\zco w=r$.  
\end{corollary}

\begin{example}
Consider the polynomial $f(x)=(x-1)^4$.   Writing $w=1+\mathfrak{D}w\in\CZ_2$, then $\kappa(\mathfrak{D}w)\le 3<4=\mu_f(1)$ and $\varphi(w)=0$. 
\end{example}

\section{Complex zeon polynomials}\label{zeon FTA}

Generally, the Fundamental Theorem of Algebra does not hold for zeon polynomials.  That is, if $\varphi(u)$ is a nonconstant complex zeon polynomial, then $\varphi(u)=0$ does not necessarily have a complex zeon root.  

For example, it is straightforward to verify that $(u-1)^2+\zeta_{\{1\}}$ has no zeon zeros.  Hence, $\varphi(u)$ need not have a zero even when $\zco(\varphi(u))$ is a nonconstant complex polynomial (which has a zero by the FTA).

On the other hand, it was shown in \cite{DollarStaples} that every invertible zeon element with positive scalar part has a real square root.  Extending to complex zeons,  $(u-1)^2+(a+b\zeta_{\{1\}})$ has zeros for all $a,b\in\mathbb{C}$ where $a\ne 0$.  In this particular example, one sees that $u=i\sqrt{a}+\displaystyle\frac{i b}{2\sqrt{a}}\zeta_{\{1\}}$ is a zero.  

However, as seen in \cite{HaakeStaples}, a zeon quadratic polynomial has solutions if and only if its discriminant has a square root.   The zeon quadratic formula established there for the real case is recalled here and extended to complex zeons.

\begin{theorem}[Zeon Quadratic Formula]\label{ZQF}
Let $\varphi(u)=\alpha u^2+\beta u+\gamma$ be a quadratic function with zeon coefficients from ${\CZ_n}$, where $\zco \alpha\ne 0$.  Let  $\Delta_{\varphi}=\beta^2-4\alpha\gamma$ denote the zeon discriminant of $\varphi$.  The zeros of $\varphi$ are given by 
\begin{equation*}
\varphi^{-1}(0)=\left\{\frac{\alpha^{-1}}{2}(w-\beta): w^2=\beta^2 - 4\alpha\gamma\right\}.
\end{equation*}
In particular, 
\begin{enumerate}[(i.)]
\item If $\Delta_{\varphi}=0$, then the zeros of $\varphi$ are given by $u=-\alpha^{-1}\beta/2+\eta$ for any $\eta\in{\CZ_n}$ satisfying $\eta^2=0$.  \label{p1}
\item If $\zco\Delta_{\varphi}\ne 0$, then $\varphi(u)=0$ has two distinct solutions.\label{p2}
\item If $\Delta_{\varphi}\ne 0$ is nilpotent and $\varphi(u)=0$ has a solution, then it has infinitely many solutions.  \label{p3}
\end{enumerate}
\end{theorem}

\begin{proof}
Writing \begin{equation*}\alpha u^2+\beta u+\gamma = \frac{\alpha^{-1}}{4}((2\alpha u + \beta)^2 - (\beta^2 - 4\alpha\gamma)),
\end{equation*}
it follows that the zeros of the polynomial are precisely the elements \begin{equation*}
\varphi^{-1}(0)=\left\{\frac{\alpha^{-1}}{2}(w-\beta): w^2=\beta^2 - 4\alpha\gamma\right\}.
\end{equation*}  Thus, the problem of finding zeros of the polynomial is essentially reduced to finding square roots of the discriminant.  

To establish Case (\ref{p1}.),  set $u_0=-\alpha^{-1}\beta/2$ and let $\eta\in \CZ_n$ such that $\eta^2=0$.   Letting $w=2\alpha \eta$, it follows that $w^2=0$, and \begin{eqnarray*}
u_0+\eta&=&u_0+\frac{\alpha^{-1}}{2}w\\
&=&-\frac{\alpha^{-1}}{2}\beta+\frac{\alpha^{-1}}{2}w\\
&=&\frac{\alpha^{-1}}{2}(w-\beta)\in\varphi^{-1}(0).
\end{eqnarray*}

Case (\ref{p2}.) is established by noting that any invertible zeon element has two distinct invertible square roots, as established in the corollary to Theorem \ref{zeon roots}. 

In Case (\ref{p3}.), existence of one square root of a nilpotent element guarantees infinitely many square roots because $(w+a \zeta_{[n]})^2=w^2$ for any $a\in\mathbb{C}$ and any nilpotent $w\in \CZ_n$. 
\end{proof}

\begin{example}
The zeon polynomial $\varphi(u)=1+(\zeta_{\{1,2\}}-2)u+(1+\zeta_{\{1\}})u^2$ over $\CZ_2$, which has discriminant \begin{eqnarray*}
\Delta_\varphi&=&(\zeta_{\{1,2\}}-2)^2-4(1+\zeta_{\{1\}})1\\
&=&-4\zeta_{\{1,2\}}-4\zeta_{\{1\}},
\end{eqnarray*}
is thus seen to have no zeros\footnote{Among other properties, squares of zeon elements must be of even grade.}, even though all of its coefficients are invertible.  
\end{example}

In each of the previous examples, $\zco(\varphi(u))$ is a complex polynomial having a multiple root. In fact, when $\zco(\varphi(u))$ has a multiple root $w_0$, $\varphi(u)$ may or may not have a zero $w$ satisfying $\zco(w)=w_0$.   For this reason, discussion is now restricted to simple zeros of polynomials with zeon coefficients.

\subsection{Spectrally Simple Zeros of Complex Zeon Polynomials}

The main result of this section is a useful version of the Fundamental Theorem of Algebra for zeon polynomials.  It not only guarantees the existence of a simple invertible zeon zero of $\varphi(u)$ when the complex polynomial $f(u)=\zco(\varphi(u))$ has a simple invertible zero, but it provides an algorithm for computing such a zero.  

Letting $\varphi(u)$ be a nonconstant monic zeon polynomial, we consider $\lambda\in{\CZ_n}$ to be a {\em simple} zero of $\varphi$ if $\varphi(u)=(u-\lambda)g(\lambda)$ for some zeon polynomial $g$ satisfying $g(\lambda)\ne 0$.

\begin{definition}
A simple zero $\lambda\in{\CZ_n}$ of $\varphi(u)$ is said to be a {\em spectrally simple} if $\zco \lambda$ is a simple zero of the complex polynomial $f=\zco\varphi$.
\end{definition}  

With the notion of spectrally simple zeros in hand, a fundamental theorem of algebra for zeon polynomials can be presented.

\begin{theorem}[Fundamental Theorem of Zeon Algebra]\label{zeon fta}
Let $\varphi(u)$ be a monic polynomial of degree $m$ over $\CZ_n$, and let $f(u)=\zco(\varphi(u))$ be the complex polynomial induced by $\varphi$.  If $\lambda_0\in\mathbb{C}$ is a simple zero of $\zco(\varphi(u))$, let $g(u)$ be the unique complex polynomial satisfying $\zco(\varphi(u))=(u-\lambda_0)g(u)$. It follows that $\varphi(u)$ has a simple zero $\lambda$ such that $\zco\lambda = \lambda_0$.  In particular, for $1\le k\le n$, the grade-$k$ part of $\lambda$ (denoted $\lambda_k$) is given by 
\begin{equation*}
\lambda_k = \frac{1}{g(\lambda_0)}\left< \varphi\left(\sum_{i=0}^{k-1} \lambda_i\right)\right>_k.
\end{equation*}
Moreover, such a zero $\lambda$ is unique.
\end{theorem}

\begin{proof}
Assuming that $\lambda_0$ is a simple zero of $\zco(\varphi(u))$, one can write \begin{equation*}
\zco(\varphi(u))=(u-\lambda_0)g(u)
\end{equation*}
where $g$ is a complex polynomial of degree $m-1$ satisfying $g(\lambda_0)\ne 0$.

Applying the polynomial division algorithm, there exist zeon polynomials $q$ and $r$ such that  
\begin{equation*}
\varphi(u)=(u-\lambda_0)q(u)+r(u)
\end{equation*}
where $r(u)$ is either zero or a zeon constant. If $r(u)=0$, then $\lambda_0$ is a complex zero of $\varphi(u)$, and the proof is complete.  Otherwise, it follows that 
\begin{equation*}
\varphi(\lambda_0)=r(\lambda_0)
\end{equation*}
and further, that $\zco (r(\lambda_0))=g(\lambda_0)\ne 0$, so that $r(\lambda_0)=r(u)$ is invertible.  

For convenience, set $\nu=-1/g(\lambda_0)$, and let $w_\natural$ denote the minimal grade part of $w\in\CZ_n$; i.e., $w_\natural= \langle w\rangle_{\natural u}$.  Replacing $\lambda_0$ with $\lambda_0 + \nu \varphi(\lambda_0)_\natural$ in the evaluation, we obtain
\begin{eqnarray*}
\varphi(\lambda_0 + \nu \varphi(\lambda_0)_\natural)&=&
\nu\varphi(\lambda_0)_\natural q(\lambda_0 +  \nu \varphi(\lambda_0)_\natural)+r(\lambda_0+\nu\varphi(\lambda_0)_\natural)\\
&=&\nu\varphi(\lambda_0)_\natural\zco(q(\lambda_0 +  \nu \varphi(\lambda_0)_\natural))+r(\lambda_0)\\
&&+\nu\varphi(\lambda_0)_\natural\zdu(q(\lambda_0 +  \nu \varphi(\lambda_0)_\natural).
\end{eqnarray*}
Letting $\ell_0=\natural \varphi(\lambda_0)$, 
\begin{eqnarray*}
\langle\varphi(\lambda_0 + \nu \varphi(\lambda_0)_\natural)\rangle_{\ell_0}
&=&\nu\varphi(\lambda_0)_\natural\zco(q(\lambda_0 +  \nu \varphi(\lambda_0)_\natural))+\langle r(\lambda_0)\rangle_{\ell_0}\\
&=&\nu\varphi(\lambda_0)_\natural g(\lambda_0)+\varphi(\lambda_0)_\natural\\
&=&-\varphi(\lambda_0)_\natural +\varphi(\lambda_0)_\natural\\
&=&0.
\end{eqnarray*}
It follows that either $\varphi(\lambda_0 + \nu \varphi(\lambda_0)_\natural)=0$ or $\natural \varphi(\lambda_0 + \nu \varphi(\lambda_0)_\natural)>\natural \varphi(\lambda_0)$.  Proceeding inductively by setting \begin{eqnarray*}
\lambda_1&=&\nu \varphi(\lambda_0)_\natural,\\
\lambda_k&=&\nu \varphi(\lambda_0+\lambda_1+\cdots+\lambda_{k-1})_\natural, \,\,(k>1)
\end{eqnarray*}
and setting $\ell_j=\natural \varphi(\lambda_0+\cdots+\lambda_{j-1})$ for $1\le j$,  it follows that 
\begin{eqnarray*}
\langle \varphi(\lambda_0 + \lambda_1+\cdots+\lambda_{k})\rangle_{\ell_{k-1}}&=&\nu\varphi(\lambda_0+\cdots+\lambda_{k-1})_\natural\zco(q(\lambda_0 +\cdots+\lambda_{k-1}))\\
&&+\langle r(\lambda_0+\cdots+\lambda_{k-1})\rangle_{\ell_{k-1}}\\
&=&\nu\varphi(\lambda_0+\cdots+\lambda_{k-1})_\natural g(\lambda_0)\\
&&+\varphi(\lambda_0+\cdots+\lambda_{k-1})_\natural\\
&=&-\varphi(\lambda_0+\cdots+\lambda_{k-1})_\natural +\varphi(\lambda_0+\cdots+\lambda_{k-1})_\natural\\
&=&0.
\end{eqnarray*}
Hence, $\varphi(\lambda_0 + \lambda_1+\cdots+\lambda_{k})=0$ or \begin{equation*}
\natural\varphi(\lambda_0 + \lambda_1+\cdots +\lambda_{k})>\natural\varphi(\lambda_0 + \lambda_1+\cdots +\lambda_{k-1}).
\end{equation*}
Since the number of generators of $\CZ_n$ is $n$, the process terminates with $\lambda=\displaystyle\sum_{k} \lambda_k$ such that $\varphi(\lambda)=0$.
   
Uniqueness of the zero $\lambda$ is trivial if $\lambda\in\mathbb{C}$.  Uniqueness of nontrivial zeon zero $\lambda$ is verified as follows.  Suppose $w=\lambda_0+\zdu w$ and $v=\lambda_0+\zdu v$ satisfy $\varphi(w)=\varphi(v)=0$.  Since $\varphi(w)=0$, one can write $\varphi(u)=(u-w)q(u)$ for monic zeon polynomial $q(u)$.   Further, since $\lambda_0$ is a simple zero of $\zco(\varphi(u))$, it follows that $\zco(q(\lambda_0))\ne 0$.  Now $\varphi(v)=0$ implies 
\begin{equation*}
(v-w)q(v)=0,
\end{equation*}
where $\zco(q(v))\ne 0$ implies that $q(v)$ is an invertible element of $\CZ_n$.  Hence, $v-w=0$.
\end{proof}

\begin{remark}
The requirement that $\varphi(u)$ be monic can be relaxed by simply requiring the leading coefficient $\alpha_m$ of $\varphi$ to be invertible.  In that case, the proposition is applied to the  monic polynomial ${\alpha_m}^{-1}\varphi(u)$.
\end{remark}

\begin{corollary}
Let $\varphi(u)$ be a complex zeon polynomial of degree $m\ge 1$. If $\zco(\varphi(u))$ is a nonconstant complex polynomial whose zeros are all simple, then $\varphi(u)$ has exactly $m$ complex zeon zeros.  In this case, we can say $\varphi$ {\em splits} over $\mathbb{C}\Z_n$.
\end{corollary}

\begin{proof}
By Theorem \ref{zeon fta}, $\varphi$ has a zero of the form $w=w_0+\zdu w$, where $w_0\ne 0$ is a zero of the complex polynomial $\zco(\varphi(u))$.  If the zeros of $\zco(\varphi(u))$ are simple, then there exist $m$ such zeros.
\end{proof}

\begin{example}
Observe that for any invertible $\alpha\in \CZ_n$, the zeon polynomial $\varphi_\alpha(u)=u^k+\alpha$ has exactly $k$ distinct complex zeon roots.
\end{example}

\subsubsection{Computing spectrally simple zeon zeros}

Presented here is an algorithm for finding an invertible, spectrally simple zeon zero $\lambda$ of polynomial $\varphi(u)$.

\begin{algorithm}[ht]
\SetKwInOut{Input}{input}\SetKwInOut{Output}{output} 
\Input{Zeon polynomial $\varphi(u)$ over $\CZ_n$ and a simple nonzero root $\lambda_0$ of the associated complex polynomial $\zco(\varphi(u))$.}
\Output{Zeon zero $\lambda$ of $\varphi(u)$ with $\zco\lambda=\lambda_0$. } 
\BlankLine
\emph{Initialize complex polynomial $g(u)$.}
\BlankLine
$g(u)\leftarrow \displaystyle\frac{\zco(\varphi(u))}{u-\lambda_0}$\;
\BlankLine
\emph{Note $g(u)$ satisfies $\varphi(u)=(u-\lambda_0)g(u)$, where $g(\lambda_0)\ne 0$.}
\BlankLine
$\xi\leftarrow\, \displaystyle\varphi(\lambda_0)_\natural/g(\lambda_0)$\;
$\lambda\leftarrow\lambda_0-\xi$;
\BlankLine
\While{$0<\natural \xi\le n$ }{
\BlankLine
$\xi\leftarrow\, \displaystyle\varphi(\lambda)_\natural/g(\lambda_0)$\;
\BlankLine
$\lambda\leftarrow(\lambda-\xi)$;
}
\Return$\lambda$\;
\BlankLine
\caption{Compute spectrally simple invertible zeon zero.\label{GreedyZeonAlgorithm}} 
\end{algorithm}

\begin{example}
Consider the following zeon polynomial $\varphi(u)$ over $\Z_4$: \begin{eqnarray*}
\varphi(u)&=&u^4-6 u^3+\left(-\zeta _{\{1,2\}}-\zeta _{\{1,3\}}-\zeta _{\{1,4\}}+12\right)u^2\\
&&+\left(2 \zeta _{\{1,2\}}+2 \zeta _{\{1,3\}}+2 \zeta _{\{1,4\}}-10\right)u +3-\zeta _{\{1,2\}}-\zeta _{\{1,3\}}-\zeta _{\{1,4\}}.
\end{eqnarray*}
The associated scalar polynomial is $\zco \varphi(u)=u^4-6 u^3+12 u^2-10 u+3$, which has zeros 3 (of multiplicity 1) and 1 (of multiplicity 3).  The unique complex zeon zero $u_0$ such that $\zco(u_0)=3$ is \begin{equation*}
u_0=3+\frac{1}{2}\left(\zeta _{\{1,2\}}+\zeta _{\{1,3\}}+\zeta _{\{1,4\}}\right). 
\end{equation*}
\end{example}

\subsection{Zeon extensions of analytic functions and their inverses}

Extending results from the real case~\cite{staplesweygandt}, the domain of an analytic function $f:A\subseteq\mathbb{C}\to\mathbb{C}$ can be extended to $A\oplus {\CZ_n}^\circ$ by defining \begin{equation*}
\varphi(u)=\sum_{k=0}^n\frac{f^{(k)}(\zco u)}{k!}(\mathfrak{D}u)^k
\end{equation*}
for  $u\in \CZ_n$ with $\zco u\in A$.  The Pigeonhole Principle guarantees that $(\mathfrak{D}u)^k=0$ for all $k>n$.  Hence, for fixed $z_\varnothing\in A$, any analytic function $f:A\to\mathbb{C}$ has a zeon polynomial form centered at $z_\varnothing=\zco u$: \begin{equation*}
\varphi_{z_\varnothing}(u)=\alpha_n (u-z_\varnothing)^n+\alpha_{n-1}(u-z_\varnothing)^{n-1}+\cdots+\alpha_1(u-z_\varnothing)+\alpha_0, 
\end{equation*}
where the complex coefficients $\alpha_k$ are determined by $\alpha_k=\displaystyle\frac{f^{(k)}(z_\varnothing)}{k!}$.  This polynomial representation is valid for all  $u\in\CZ_n$ such that $\zco u=z_\varnothing$.  In turn, each such polynomial  defines a spectral equivalence class of zeons; i.e., $z_1$ and $z_2$ are in the same class if and only if $\zco z_1=\zco z_2$.

It follows that to evaluate $\varphi(u)$ for $u=z_\varnothing+\zdu u$, one may simply evaluate $\varphi_{z_\varnothing}(\zdu u)$.  Moreover, to compute the inverse image of a complex zeon of the form $w=\varphi(z_\varnothing)+\zdu w$, one need only seek zeros of the polynomial $\varphi_{z_\varnothing}(u)-\zdu w$.  

\begin{lemma}
Let $\varphi(u):A\oplus {\CZ_n}^\circ$ be the zeon extension of an analytic complex function $f:A\subseteq\mathbb{C}\to\mathbb{C}$, and let $w\in f(A)\oplus\CZ_n^\circ$.  Let $z_\varnothing\in A$ denote a preimage of $w_\varnothing=\zco w$ under $f$.  If the complex polynomial $\zco (\varphi_{z_\varnothing}(u)-w)$ has a simple zero at $z_\varnothing$, then the spectrally simple zeon zero $\lambda$ of the zeon polynomial $\psi(u)=\varphi_{z_\varnothing}(u)-w$ satisfies\begin{equation*}
\varphi(\lambda)=w.
\end{equation*}
In other words, $\lambda=\varphi^{-1}(w)$ is the preimage of $w$ under $\varphi$.   
\end{lemma}

\begin{example}
As a simple example, we wish to compute the preimage of $w=\frac{\sqrt{3}}{2}+\left(\zeta _{\{1,2\}}+3 \zeta _{\{1\}}-\zeta _{\{4\}}\right)$ under the zeon extension of $f(z)=\cos z$ to $\CZ_4$. The required polynomial representation is seen to be \begin{eqnarray*}
\varphi_{\pi/3}(u)-w&=&
-\zeta _{\{1,2\}}-3 \zeta _{\{1\}}+\zeta _{\{4\}}+\frac{\left(u-\frac{\pi }{6}\right)^4}{16 \sqrt{3}}+\frac{1}{12} \left(u-\frac{\pi }{6}\right)^3\\
&&+\frac{1}{12} (\pi -6 u)-\frac{(\pi -6 u)^2}{48 \sqrt{3}}.
\end{eqnarray*}
The scalar projection $\zco(\varphi_{\pi/3}(u)-w)$ has a simple zero at $u_0=\frac{\pi}{3}$, and the corresponding unique complex zeon zero is numerically computed to be \begin{equation*}
\lambda=0.523599-2. \zeta _{\{1,2\}}+20.7846 \zeta _{\{1,4\}}+6.9282 \zeta _{\{1,2,4\}}-6. \zeta _{\{1\}}+2. \zeta _{\{4\}}.
\end{equation*}
The preimage is verified by computing $\varphi(\lambda)$:
\begin{equation*}
\varphi(\lambda)=0.866025+1. \zeta _{\{1,2\}}+3. \zeta _{\{1\}}-1. \zeta _{\{4\}},
\end{equation*}
where $0.866025\approx\frac{\sqrt{3}}{2}$, as desired.
\end{example}

As a corollary, it follows that when $\varphi$ is the zeon extension of an invertible analytic function $f$, the zeon polynomial $\varphi_{z_\varnothing}(u)-w$ will have a unique spectrally simple zeon zero $\lambda$.

\begin{corollary}
Let $\varphi(u):A\oplus {\CZ_n}^\circ$ be the zeon extension of an invertible analytic complex function $f:A\subseteq\mathbb{C}\to\mathbb{C}$, and let $w\in f(A)\oplus\CZ_n^\circ$.  Let $z_\varnothing\in A$ denote the unique preimage of $w_\varnothing=\zco w$ under $f$. The unique spectrally simple zeon zero $\lambda$ of the zeon polynomial  $\psi(u)=\varphi_{z_\varnothing}(u)-w$ satisfies\begin{equation*}
\varphi(\lambda)=w.
\end{equation*}
In other words, $\lambda=\varphi^{-1}(w)$ is the preimage of $w$ under $\varphi$.   
\end{corollary}

\section{Non-spectrally Simple Zeon Zeros}

A little more can be said about simple zeon zeros of $\varphi$ whose scalar part is a multiple zero of $\zco \varphi$.

\begin{theorem}
Let $\varphi(u)$ be a monic polynomial over $\CZ_n$, and let $w_0\in\mathbb{C}$.  If $\varphi$ has complex zeon zeros $w_1, w_2$ satisfying $\zco w_1=\zco w_2=w_0$, then 
$\varphi$ has infinitely many zeros of the form $w=w_0+\zdu w$.
\end{theorem}

\begin{proof}
If $\varphi$ has complex zeon zeros $w_1, w_2$, then one can write $\varphi(u)=(u-w_1)(u-w_2)g(u)$, where $g(u)$ is a nonzero zeon polynomial.  Assuming satisfying $\zco w_1=\zco w_2=w_0$, it follows that for any nonzero scalar $a$,
 \begin{eqnarray*}
\varphi(w_1+a \zeta_{[n]})&=&(w_1+a \zeta_{[n]}-w_1)(w_1+a \zeta_{[n]}-w_2)g(w_1+a \zeta_{[n]})\\
&=&a\zeta_{[n]}(a\zeta_{[n]}+\zdu w_1-\zdu w_2)g(w_1+a \zeta_{[n]})\\
&=&0.
\end{eqnarray*}
\end{proof}

This leads immediately to a corollary involving multiple zeros.   If $z$ is a zeon zero of multiplicity $m\ge 2$, it follows that $\varphi(u)=(u-z)^{m}\gamma(u)$, where $\gamma(u)$ is a zeon polynomial such that $\gamma(z)\ne 0$.  It thereby follows that for any nilpotent $w\in\CZ_n$ such that $\kappa(w)\le m$, the following holds:
\begin{eqnarray*}
\varphi(z+w)&=&(z+w-z)^m\gamma(z+w)\\
&=&w^m\gamma(z+w)\\
&=&0.
\end{eqnarray*}
Thus, the following corollary is established.

\begin{corollary}
Let $\varphi(u)$ be a monic polynomial over $\CZ_n$.  If  $z\in\CZ_n$ is a complex zeon zero of $\varphi$ of multiplicity $m\ge 2$, then $\varphi$ has infinitely many complex zeon zeros zeros $u$ satisfying $\zco u=\zco z$ .
\end{corollary}

\section{Conclusion}\label{conclusion}

To summarize, a complex zeon polynomial $\varphi$ of degree $m$ whose corresponding scalar polynomial $f=\zco \varphi$ has a simple complex zero $r$ will have a unique simple complex zeon zero $u$ whose scalar part is $\zco u=r$.  If $f$ has a complex zero $r$ of multiplicity greater than $1$ and $\varphi$ has a complex zeon zero $u$ such that $\zco u=r$, then $\varphi$ has infinitely many such complex zeon zeros.    When $\varphi$ is the zeon extension of an invertible analytic function, inverse images of elements in the range of $\varphi$ can be computed using the polynomial methods developed here.  

With these basic results on zeon polynomials established, other avenues of research can be explored.  Zeon matrix polynomials and eigenvalues of zeon matrices are of particular interest in light of their known graph-theoretic applications and properties.

\end{document}